\documentclass{amsart}
\usepackage{amsmath, amssymb}
\usepackage{amstext, amscd, amsfonts}
\usepackage{mathtools}
\usepackage{cancel}
\usepackage{amsthm}
\usepackage{enumerate}
\usepackage{latexsym}
\usepackage{verbatim, epsfig}
\usepackage{dsfont}
\usepackage{bbm}
\usepackage{ifthen}
\usepackage{xcolor}
\usepackage{amsbsy} 
\usepackage{varioref} 
\usepackage{url}
\usepackage[english]{babel}  
\usepackage{microtype} 

\DeclareMathOperator{\supp}{supp}

\newtheorem{theorem}{Theorem}[section]
\newtheorem{theorem*}{Theorem}

\newtheorem{corollary}[theorem]{Corollary}

\newtheorem{lemma}[theorem]{Lemma}

\newtheorem{question}[theorem]{Question}

\theoremstyle{definition}

\newtheorem{definition}[theorem]{Definition}

\newtheorem{example}[theorem]{Example}

\newtheorem{remark}[theorem]{Remark}

\numberwithin{equation}{section}
\numberwithin{figure}{section}

\title[Multiplicity structure of invariant measures]{Multiplicity structure of preimages of invariant measures under finite-to-one factor maps}
\author{Jisang Yoo}
\address{Seoul National University, Seoul, South Korea}
\email{jisangy@kaist.ac.kr}
\thanks{This research was supported by BK21 PLUS SNU Mathematical Sciences Division. The author thanks Uijin Jung, Sujin Shin, Soonjo Hong and referees for helpful comments. This research was supported by Basic Science Research Program through the National Research Foundation of Korea(NRF) funded by the Ministry of Education(2012R1A6A3A01040839) and the National Research Foundation of Korea (NRF) grant funded by the MEST 2015R1A3A2031159.}

\begin{document}
\maketitle
\begin{abstract}
Given a finite-to-one factor map $\pi: (X, T) \to (Y, S)$ between topological dynamical systems, we look into the pushforward map $\pi_*: M(X, T) \to M(Y,T)$ between sets of invariant measures. We investigate the structure of the measure fiber $\pi_*^{-1}(\nu)$ for an arbitrary ergodic measure $\nu$ on the factor system $Y$.
We define the degree $d_{\pi,\nu}$ of the factor map $\pi$ relative to $\nu$ and the multiplicity of each ergodic measure $\mu$ on $X$ that projects to $\nu$, and show that the number of ergodic pre-images of $\nu$ is  $d_{\pi,\nu}$ counting multiplicity. In other words, the degree $d_{\pi,\nu}$ is the sum of the multiplicity of $\mu$ where $\mu$ runs over the ergodic measures in the measure fiber $\pi^{-1}_*(\nu)$.
This generalizes the following folklore result in symbolic dynamics for lifting fully supported invariant measures: Given a finite-to-one factor code $\pi: X \to Y$ between irreducible sofic shifts and an ergodic measure $\nu$ on $Y$ with full support, $\pi^{-1}_*(\nu)$ has at most $d_\pi$ ergodic measures in it, where $d_\pi$ is the degree of $\pi$.
We apply our theory of structure of measure fibers to the special case of symbolic dynamical systems. In this case, we demonstrate that one can list all (finitely many) ergodic measures in the measure fiber $\pi^{-1}_*(\nu)$.
\end{abstract}

\section{Introduction}
Under some reasonable assumptions, a classical dichotomy result on factor maps $\pi:X \to Y$ between symbolic dynamical systems (i.e. factor codes) classifies them into two categories (see Theorem~\ref{thm:fto-conditions} for a precise statement). The first category consists of finite-to-one factor codes where typical fibers have finite cardinality and the factor system $Y$ and the extension $X$ have the same entropy. The other category consists of infinite-to-one factor codes where typical fibers have infinite cardinality and the factor system $Y$ has lower entropy than $X$.

Finite-to-one factor codes are more well understood than infinite-to-one codes.
For a finite-to-one factor code $\pi$, one can associate a single number $d= d_\pi$ called the degree of $\pi$ such that the map $\pi$ is almost $d$-to-one in some sense. This is a topological analogue to a result in ergodic theory that a finite-to-one factor map between two ergodic systems is a.e. constant-to-one i.e. mod 0 isomorphic to a constant-to-one map. On the other hand, we do not have an analogue of a stronger result like Rohlin's skew-product theorem that any factor map between two ergodic systems is isomorphic to that from a skew-product (Theorem 3.18 in \cite{glasner2003ergodic}). Finite-to-one factor codes in general cannot be represented as any kind of topological skew-product.

Given a factor map $\pi: X \to Y$ between irreducible sofic shifts or subshifts of finite type, we have an induced onto map $\pi_*$ from the set of invariant probability measures on $X$ to that of $Y$.
We are interested in the structure of the measure fiber $\pi_*^{-1}(\nu)$ where $\nu$ is a fixed ergodic measure on $Y$.
A classical folklore result says that if $\pi$ is finite-to-one and $\nu$ has full support, then the number of ergodic measures in its measure fiber is bounded by the degree $d$. If we relax the full support condition of $\nu$, the number may exceed $d$ but it is still finite (see Example~\ref{ex:exceed}).
Since the ergodic measures in $\pi_*^{-1}(\nu)$ are precisely its extreme points, the measure fiber is a simplex with finitely many extreme points. Since a simplex is determined by its extreme points, knowing $\pi_*^{-1}(\nu)$ is the same as knowing all ergodic measures that project to $\nu$.

A surprising result along this line concerns the special case when $\nu$ is a (fully supported) Markov measure (which forces $Y$ to be a subshift of finite type rather than a strictly sofic shift), or more generally when $\nu$ is the unique equilibrium state (or equivalently, invariant Gibbs measure) of some regular potential function defined on a mixing subshift of finite type $Y$.
In \cite{Tuncel1981conditional}, Tuncel proved the following theorem: if $\nu$ is as described, then it lifts uniquely through the finite-to-one factor code $\pi$. In other words, there is only one measure $\mu$ in $\pi_*^{-1}(\nu)$, even when $d > 1$. The unique lift $\mu$ is easily described from $\nu$ and $\pi$. In particular, if $\nu$ is Markov, so is $\mu$.
This case is strictly a special case: Every such equilibrium state is ergodic and fully supported, but not every fully supported ergodic measure on $Y$ is an equilibrium state of a regular potential function. In fact, such equilibrium states come with very strong mixing properties that are typically not shared by arbitrary fully supported ergodic measures \cite{R-ThermoFormal}.

Tuncel's result can be thought of as a generalization of an earlier and more easily proved result that finite-to-one factor codes preserve maximal measures (i.e. measures of maximal entropy). That is, $\pi_* \mu_0 = \nu_0$ where $\mu_0, \nu_0$ are the unique maximal measures on $X, Y$ respectively.

In this paper, we develop a theory of the structure of $\pi_*^{-1}(\nu)$ for the general case when $\nu$ is an arbitrary ergodic measure.
Even when we are mainly interested in Markov measures, the general case can be of interest for the following reason. When we are given a Markov measure $\mu$ on $X$ that we want to investigate, we might find a convenient finite-to-one factor code $\pi$ on $X$ to exploit, but the image measure $\nu := \pi_* \mu$ does not have to be a Markov measure. In fact, the sofic measure (the image of a Markov measure, a.k.a. stationary hidden Markov chain) in general may not even be an equilibrium state of any regular potential, let alone a Markov measure. (See \cite{Boyle-Petersen-2009-hidden-markov-symbolic} for examples and introduction.)

Finite-to-one factor codes are also relevant for studying the evolution of measures under surjective cellular automata. When $X=Y={\mathcal A}^{\mathbb Z}$ for some alphabet $\mathcal A$, any factor code between $X, Y$ is a 1-dimensional surjective cellular automaton and $\pi$ is finite-to-one because $X$ and $Y$ have the same entropy.

For infinite-to-one factor codes, there are usually infinitely many ergodic measures in $\pi_*^{-1}(\nu)$. In this case, people are more interested in the finitely many entropy-maximizing ergodic measures within $\pi_*^{-1}(\nu)$, which are called measures of maximal relative entropy \cite{all2013classdegrelmaxent, PQS-MaxRelEnt}. The problem of lifting invariant measures through finite-to-one factor codes has a close connection with the problem of lifting through infinite-to-one factor codes. Understanding the former can help understanding the latter because of the following two reduction results. One is that in each infinite-to-one factor code $\pi:X \to Y$, under some reasonable assumptions, one can always find a subshift $X_1 \subset X$ on which the induced factor code $\pi':X_1 \to Y$ is finite-to-one \cite{MPW1984transmission}. The other is that the infinite-to-one factor code $\pi:X \to Y$ can be decomposed into the composition of two factor codes $\pi_1: X \to M$ and $\pi_2: M \to Y$ where $\pi_2$ is finite-to-one and $\pi_1$ is a class degree one factor code (this is to be published in a subsequent paper). We hope that the theory of finite-to-one measure fiber structure can shed new light on the study of measures of maximal relative entropy and relative thermodynamic formalism in general.

As part of the structure of the measure fiber $\pi_*^{-1}(\nu)$, for each ergodic measure $\mu$ in it, we define the multiplicity of $\mu$ over its image $\nu = \pi_* \mu$. A simple example to motivate the notion of multiplicity is the following from \cite{Walters-Relative}.
\begin{example}\label{first-example}
  Let $X = Y = \{0,1\}^{\mathbb Z}$ be two copies of the full two shift. Define $\pi: X\to Y$ by $\pi(x) = y$ where $y_i = x_i + x_{i+1} \pmod 2$ for $i \in \mathbb Z$. If this map is seen as an endomorphism of the full shift rather than as a factor code, then this is nothing but the rule 102 cellular automaton (which induces Ledrappier's three dot example). The factor code $\pi$ is 2-to-1. For each $0 < p < 1$, define $\mu_p$ to be the Bernoulli product measure on $X$ with probability $p$ for value $1$ and $1-p$ for $0$. Let $\mu'_{p} = \mu_{1-p}$. Then $\mu_p$ and $\mu'_p$ project to a common measure $\nu_p = \pi(\mu_p) = \pi(\mu'_p)$ on $Y$. The two measures $\mu_p, \mu'_p$ are distinct unless $p = \frac12$.
  After we define multiplicity in a later section, we will see that the measure $\mu_{\frac12} = \mu'_{\frac12}$ is of multiplicity
  two w.r.t. $\pi$ over $\nu_{\frac12}$. And we will see that for $p\ne \frac12$, the measures $\mu_p, \mu'_p$ have multiplicity one over $\nu_p$.
\end{example}

We also introduce the notion of degree joining which is an essentially unique object obtained by joining all ergodic lifts $\mu$ of $\nu$ counted with multiplicity. In order to take shortcuts by relying on the theory of joinings, we define and construct the degree joining before we define multiplicity. Then we build a general theory of multiplicity by relying on the constructed degree joining.

By exploiting the multiplicity structure of finite-to-one measure fibers, we are able to demonstrate that in many cases, given an ergodic measure $\mu$ on $X$, as soon as one knows a concrete way to list all points in the fiber $\pi^{-1}(\pi(x))$ from a given point $x \in X$, one also has a way to list all ergodic measures in $\pi^{-1}(\pi(\mu))$ and count the number of them.
In particular, we can build an example of a 5-to-1 factor code such that in a broad class of (fully supported) $\nu$, the number of ergodic measures in $\pi^{-1}(\nu)$ is strictly between 1 and the degree 5. Previously there has been no tools to establish such examples.

Next sections are organized as follows. In Section~\ref{sec:back}, we fix notations and elementary definitions. In Section~\ref{sec:deg}, we define measure theoretical degree and canonical lift. In Section~\ref{sec:deg-join}, degree joinings are introduced. In Section~\ref{sec:mult}, the multiplicity of ergodic measures over finite-to-one factor maps is defined. In Section~\ref{sec:ex}, we demonstrate examples of calculating measure fibers using degree joinings in case of cellular automata. In Section~\ref{sec:deg-symb}, the special case of symbolic dynamics is further investigated and some irregular examples involving measures without full support are mentioned.
\section{Background}\label{sec:back}

Unless stated otherwise, a \emph{topological dynamical system} (TDS for short) here means a compact metric space $X$ equipped with a self homeomorphism $T: X \to X$ on it, and a \emph{shift of finite type} (SFT) means a (one-dimensional) two-sided shift of finite type with a finite alphabet. In particular, we only deal with invertible systems. Shift spaces and sofic shifts are also assumed to be two-sided and with a finite alphabet. \emph{Irreducible} shift spaces mean shift spaces that are forward transitive. We remark that converting between invariant measures on one-sided shift spaces and those on two-sided shift spaces is straightforward, and therefore the two-sided condition is a minor technical assumption.

For a point $x = (x_i)_{i\in \mathbb Z}$ in a shift space $X$ and indices $i \le j$, we denote by $x_{[i,j]}$ the word $x_i x_{i+1} x_{i+2} \cdots x_j$. The shift map $\sigma_X: X \to X$ is defined by
\[ y=\sigma_X(x) \iff y_i = x_{i+1} \quad (\forall i \in \mathbb Z). \]
and will be denoted by $\sigma$ without the subscript if there is no confusion.

The topological entropy of a topological dynamical system $(X, T)$ is denoted by $h(X,T)$, or just $h(X)$ if $T$ is understood.
A subset $X_0 \subset X$ is said to be a \emph{subsystem} of $(X, T)$ if it is non-empty, closed and $T$-invariant (i.e., $TX_0 = X_0$). A subsystem is said to be \emph{proper} if it is a proper subset of the ambient space $X$.

When we say $\mu$ is a \emph{measure} on $X$, we mean that $\mu$ is a Borel probability measure on it. We denote by $M(X)$ the set of all measures on $X$. $M(X)$ is a compact metric space under the weak star topology.
If $X_0 \subset X$ is measurable and $\mu(X_0)=1$, then $\mu$ is said to be \emph{supported} on $X_0$. We may identify measures on $X$ that are supported on $X_0$ with measures on $X_0$.
For $\mu \in M(X)$, $\supp(\mu) \subset X$ denotes the topological support of the measure $\mu$, i.e., the smallest closed set of full measure w.r.t. $\mu$. The topological support of any invariant measure on $X$ is a subsystem of $X$.
Given a topological dynamical system $(X, T)$, the set of all ergodic measures on it will be denoted by $E(X, T)$ or just $E(X)$ if the action $T$ is understood. If the topological support of $\mu\in E(X, T)$ is $X$, we say $\mu$ is \emph{fully supported} or has \emph{full support}.
For $\mu \in E(X, T)$, a point $x \in X$ is called a $\mu$-\emph{generic} point if the forward averages $\frac1N \sum_{n=1}^{N} T^n \delta_x$ converge to $\mu$, where $\delta_x$ denotes the point mass at $x$ (see \cite{glasner2003ergodic}).

Each $\mu\in E(X)$ gives rise to an (abstract) measure preserving system $(X, T, \mu)$ as in ergodic theory after forgetting the topology on $X$ but keeping the Borel sigma-algebra. Ergodic measure preserving systems will be called \emph{ergodic systems} for short.

A \emph{factor map} is a continuous onto map between two topological dynamical systems that commutes with the associated homeomorphisms.
A \emph{factor code} is a factor map between two shift spaces. When we say $\pi: X\to Y$ is a factor code on an SFT $X$, it is therefore assumed that $Y$ is the image of $X$ under $\pi$ (and hence $Y$ is a sofic shift space).
Given a factor map $\pi: (X, T) \to (Y,S)$ between topological dynamical systems, we will say a measure $\mu$ on $X$ is a \emph{lift} or \emph{preimage} of a measure $\nu$ on $Y$ if the (pushforward) image of $\mu$ under $\pi$ is $\nu$, i.e., if $\pi_*\mu = \nu$. For brevity, the pushforward map $\pi_*: M(X) \to M(Y)$ is denoted by $\pi$ when there is no confusion. In other words, we write $\pi\mu$ for $\pi_* \mu = \mu \circ \pi^{-1}$.

For more background on ergodic theory and theory of joinings, see \cite{glasner2003ergodic}. For background on factor codes for symbolic dynamics, see \cite{LM}.

\section{Measure-theoretical degree and canonical lift}\label{sec:deg}
In this section, we define the notion of degree over an arbitrary ergodic measure on a factor system. This extends the classical notion of degree of finite-to-one factor codes.
For this, we need the following lemma.

\begin{lemma}
  Let $(X,T)$ and $(Y,S)$ be topological dynamical systems and $\pi: X\to Y$ a factor map. Then the map $F: Y \to \{1,2,\dots\} \cup \{\infty\}$ defined by $y \mapsto |\pi^{-1}(y)|$ is constant a.e. with respect to each ergodic measure $\nu$ on $Y$. (The constant may depend on $\nu$.)
\end{lemma}
\begin{proof}
  We do not know if $F$ is Borel-measurable, but we can show that it is universally measurable, i.e., $F$ is measurable w.r.t. every measure on $Y$.
  Recall that a subset of a Polish space is said to be an analytic set if it is the image of a Borel subset of another Polish space under a Borel-measurable map, and that any analytic subset of a Polish space is universally measurable. See \cite{glasner2003ergodic} p.~52 or \cite{kechris2012classical} p.~155 for these facts.

  For each $k \in \mathbb N$, the superlevel set $ \{y\in Y: |\pi^{-1}y| \ge k\}$ is the projection to $Y$ of a Borel subset in $X^k \times Y$, namely, the subset consisting of all $(x_1,x_2,\dots,x_k,y) \in X^k \times Y$ for which $\pi(x_i)=y$ for all $1\le i \le k$ and $x_i \neq x_j$ for all $1\le i < j \le k$. Therefore the superlevel set is an analytic subset of $Y$, and hence a universally measurable set. It follows that the map $F$ is universally measurable.

  Recall that a measurable function defined on an ergodic system is a.e. constant if the function is invariant (w.r.t. the ergodic action).
  Since the map $F$ is invariant with respect to the action $S$, it must be constant a.e. with respect to each ergodic measure on $Y$.
\end{proof}
Recall that we are assuming invertibility. The lemma fails in general for non-invertible systems because $F$ is not $S$-invariant in such cases. (A non-invertible counter-example is with $X=Y$ being the one-sided golden mean shift and $\pi = \sigma_X$.)

For each $\nu \in E(Y)$, we define the \emph{degree} of $\nu$ relative to $\pi$ to be the unique number $d \in \{1,2,\dots\} \cup \{\infty\}$ such that for $\nu$-a.e. $y \in Y$, there are precisely $d$ points in the fiber $\pi^{-1}(y)$. We will denote this number by $d_{\pi,\nu}$, and if $\pi$ is understood, by $d_\nu$.

If a factor map $\pi: X \to Y$ has the property that $d_{\pi,\nu} = d_{\pi,\nu'}$ whenever $\nu$ and $\nu'$ are \emph{fully supported} ergodic measures on $Y$, then it makes sense to define the \emph{degree of the factor map} to be the common value $d_{\pi,\nu}$ and denote it by $d_\pi$. This measure-theoretical definition generalizes the classical definition (see Section~\ref{sec:deg-symb}) of degree of finite-to-one factor codes on irreducible SFTs and sofic shifts: recall that if $\pi: X \to Y$ is a finite-to-one factor code on an irreducible sofic shift, then its degree is defined to be the unique number $d \in \mathbb N$ such that $|\pi^{-1}(y)| = d$ for all doubly transitive points $y \in Y$ (points $y$ whose forward orbits and backward orbits are dense). The two definitions of degree are consistent because the set of doubly transitive points in $Y$ has full measure with respect to each fully supported ergodic $\nu$ on $Y$ (Lemma~\ref{lem:dt-full-in-full}). 
Even when a factor code has a finite degree, the degree of an arbitrary (not necessarily fully supported) ergodic measure on $Y$ may be different (see Example~\ref{ex:diff}).

We say $(X, Y, \pi, \nu)$ is a \emph{factor quadruple} if $\pi: X \to Y$ is a factor map between two topological dynamical systems and $\nu \in E(Y)$. Note that a factor quadruple always has degree, namely $d_{\pi,\nu}$, whether finite or infinite. If the degree $d$ is finite, it makes sense to say the factor quadruple is $d$-to-one almost everywhere.

Next, we introduce the notion of canonical lift of an ergodic measure under an a.e. finite-to-one factor map.

\begin{lemma}\label{lem:fiber-is-measurable}
  Let $\pi: X \to Y$ be a Borel-measurable map between Polish spaces. Let $A\subset X$ be a Borel subset. Then the map $F_A: Y \to \{0,1,2,\dots\} \cup \{\infty\}$ defined by $$y \mapsto |\pi^{-1}(y) \cap A|$$ is universally measurable.
\end{lemma}
\begin{proof}
  For each $k \in \mathbb N$, the set $$ \{y\in Y: |\pi^{-1}y \cap A| \ge k\}$$ is the projection of a Borel subset in $X^k \times Y$, namely, the subset consisting of all $(x_1,x_2,\dots,x_k,y) \in X^k \times Y$ for which $\pi(x_i)=y$ and $x_i \in A$ for all $1\le i \le k$ and $x_i \neq x_j$ for all $1\le i < j \le k$, and therefore this set is an analytic subset of $Y$, and hence a universally measurable set. It follows that the map $F_A$ is universally measurable.
\end{proof}

\begin{theorem}\label{thm:canonical-lift-exists}
  Let $(X, Y, \pi, \nu)$ be a factor quadruple with finite degree $d := d_{\pi,\nu} < \infty$.
  Then there is a (not necessarily ergodic) invariant measure $\mu$ on $X$ such that $\pi\mu=\nu$ and that the disintegration $\{\mu_y\}_{y \in Y}$ of $\mu$ over $Y$ has the property that $\mu_y$ is the uniform distribution on the $d$-points subset $\pi^{-1}(y) \subset X$ for $\nu$-a.e. $y \in Y$. Such $\mu$ is unique and we will call it the \emph{canonical lift} of $\nu$ and denote it by $\ell_\pi(\nu)$.
\end{theorem}
\begin{proof}
  (Existence)
  For each Borel measurable $A\subset X$, we define $$\mu(A) = \frac{\int_Y F_A(y) d\nu(y)}{d}.$$
  This is well defined because of the previous lemma and it is easy to verify that $\mu$ is countably additive and $\mu(\emptyset)=0$ and $\mu(X)=1$.

  $\mu$ is $T$-invariant because
  \begin{align*}
    \mu(T^{-1}A) &= \frac{\int_Y F_{T^{-1}A}(y) d\nu(y)}{d}\\
    &= \frac{\int_Y F_{A}(Sy) d\nu(y)}{d}\\
    &= \frac{\int_Y F_{A}(y) d\nu(y)}{d}
  \end{align*}
  where the last equality holds because $\nu$ is $S$-invariant. It is also easy to verify $\pi\mu = \nu$.

  Let $Y_0$ be a Borel subset of $Y$ such that $\nu(Y_0)=1$ and $F_X(y) = d$ for all $y \in Y_0$. Then the map $U: Y_0 \to M(X)$ defined by requiring that $U_y$ be the uniform distribution on the $d$ points in $\pi^{-1}(y)$ is a $\nu$-measurable map by the previous lemma because $U_y(A)=\frac{F_A(y)}{d}$ for each Borel measurable $A \subset X$. This map $U$ is a disintegration of $\mu$ over Y, since
  \begin{align*}
    \mu(A) &= \frac{\int_Y F_A(y) d\nu(y)}{d}\\
    &=\int_{Y_0} U_y(A) d\nu(y)
  \end{align*}

  (Uniqueness)
  If $\mu'$ is another such measure, then
  \begin{align*}
    \mu' &= \int_Y U_y d\nu(y)\\
    &= \mu
  \end{align*}
\end{proof}

\section{Relative joinings and degree joinings}\label{sec:deg-join}
In this section, we introduce the notion of degree joining and investigate its properties.

Recall the definition of \emph{joining}: For invariant measures $\mu$ and $\mu'$ on topological dynamical systems $(X, T)$ and $(X', T')$ respectively, a measure $\lambda$ on $X \times X'$ is called a (2-fold) joining of $\mu$ and $\mu'$ if it is a $T\times T'$-invariant measure whose margins on $X$ and $X'$ are $\mu$ and $\mu'$ respectively.
 
We are interested in a relative version of the notion of joining.
Given a factor map $\pi: (X,T) \to (Y,S)$ between topological dynamical systems, define the $n$-fold (self-)fiber product
\begin{align*}
  X^n_\pi &:= \{(x_1, x_2, \dots, x_n) \in X^n : \pi(x_1) = \pi(x_2) = \dots = \pi(x_n)\}\\
  &= \bigcup_{y\in Y} \left(\pi^{-1}(y) \times \pi^{-1}(y) \times \cdots \times \pi^{-1}(y)\right).
\end{align*}
An \emph{$n$-fold $\pi$-relative joining} is an invariant measure $\lambda$ on $X^n$ for which the $n$-fold fiber product $X^n_\pi$ is a full measure set, i.e., $\lambda(X^n_\pi) = 1$.
We will call such $\lambda$ an $n$-fold relative joining if $\pi: X \to Y$ is understood from the context.
We will say that such a measure $\lambda$ is a relative joining \emph{of margins} $\mu_1, \dots, \mu_n$ \emph{over their common image} $\nu$ if $p_i\lambda = \mu_i$ for each $i$, where $p_i: X^n \to X$ is the projection to the $i$-th component, and $\pi p_i \lambda = \nu$ for some $i$ (and hence for all $i$). We will say such a measure $\lambda$ is \emph{separating} if for $\lambda$-a.e. $(x_1, x_2, \dots, x_n)$, the points $x_1, x_2, \dots, x_n$ are $n$ distinct points, i.e., $x_i \ne x_j$ whenever $1 \le i < j \le n$.

\begin{remark}
  $\pi$-relative joinings as defined here are related to the notion of joinings of ergodic systems over a common factor as usually defined in ergodic theory.
  In ergodic theory, if $(X, T, \mu)$ and $(X', T', \mu')$ are two ergodic systems and $\pi: (X, T, \mu) \to (Y, S, \nu)$ and $\pi': (X', T', \mu') \to (Y, S, \nu)$ are homomorphisms so that $(Y, S, \nu)$ is a common factor, then a joining $\lambda$ of the two ergodic systems is called a joining of $(X, T, \mu)$ and $(X', T', \mu')$ over $(Y, S, \nu)$ if the fiber product $\{(x,x'): \pi x = \pi' x'\} \subset X \times X'$ is a full measure set w.r.t. $\lambda$ (see \cite{glasner2003ergodic}).
  A difference in our setting is that two topological dynamical systems and a factor map between them are fixed.
  Let $\pi: (X,T) \to (Y, S)$ be a factor map between topological dynamical systems.
  If $\lambda$ is a $\pi$-relative joining of $\mu_1, \dots, \mu_n \in E(X, T)$ over $\nu \in E(Y,S)$, then $\lambda$ is also a joining of the collection of $n$ ergodic systems $(X, T, \mu_i)$, $1\le i \le n$, over a common factor $(Y,S, \nu)$.
\end{remark}

\begin{lemma}\label{lem:erg-decomp}
  Let $n$ be a positive integer and let $(X, Y, \pi, \nu)$ be a factor quadruple.
  If $\lambda$ is an $n$-fold relative joining over $\nu$, then almost every ergodic component of $\lambda$ is an $n$-fold relative joining over $\nu$. In other words, if $\lambda = \int \lambda' d\rho(\lambda')$ is the ergodic decomposition of $\lambda$, then $\lambda'$ is a relative joining over $\nu$ for $\rho$-a.e. $\lambda'$.
\end{lemma}
\begin{proof}
  It is easy to verify that almost every ergodic component of a relative joining is a relative joining.
  It only remains to show that ergodic decomposition preserves the image $\nu$. Note that the ergodic decomposition of $\lambda$ induces an ergodic decomposition of $\nu$ in the form of $$\nu = \pi p_1 \lambda = \int \pi p_1 \lambda' d\rho(\lambda')$$ but since $\nu$ is already ergodic, the induced decomposition must be trivial. Therefore, $\pi p_1 \lambda' = \nu$ for almost every $\lambda'$ and hence $\lambda'$ is a relative joining over $\nu$.
\end{proof}

Now we are ready to define and prove the existence of a \emph{degree joining}, which is a particular way of joining together all ergodic pre-images of $\nu$.

\begin{definition}
  Let $(X, Y, \pi, \nu)$ be a factor quadruple with finite degree $d := d_{\pi,\nu} < \infty$. A measure on $X^d$ is a \emph{degree joining} over $\nu$ with respect to $\pi$ if it is a $d$-fold ergodic separating relative joining over $\nu$.
\end{definition}
\begin{theorem}\label{thm:degree-joining-exists}
  For each factor quadruple with finite degree, there exists a degree joining for the quadruple.
\end{theorem}
\begin{proof}
  Let $(X, Y, \pi, \nu)$ be a factor quadruple with finite degree $d := d_{\pi,\nu} < \infty$.
  Let $\mu := \ell_\pi(\nu)$ be the canonical lift of $\nu$.
  Let $\lambda$ be the $d$-fold relatively independent joining of $\mu$ over $\nu$, i.e., $\lambda$ is the measure whose disintegration over $Y$ is given by
  $$\lambda_y = \mu_y \otimes \mu_y \otimes \dots \otimes \mu_y .$$
  (See \cite{glasner2003ergodic} Chapter 6 for basic properties of relatively independent joinings.)

  The measure $\lambda$ is a relative joining over $\nu$ and by the previous lemma, almost every ergodic component of $\lambda$ is an ergodic relative joining over $\nu$.
  Since $\mu_y$ is a uniform distribution on $d$ points, we have
  \begin{align*}
    \lambda_y(Z) &= \frac{d-1}{d} \cdot \frac{d-2}{d} \cdots \frac{1}{d}\\
    &= \frac{d!}{d^d} > 0,
  \end{align*}
  where $Z$ is the set of all $(x_1,\dots,x_d) \in X^d$ such that $x_i \neq x_j$ for all $1 \le i < j \le d$. In particular, we have $\lambda(Z)>0$.

  Let $\lambda = \int \lambda' d\rho(\lambda')$ be the ergodic decomposition of $\lambda$. Then, since
  $$ 0 < \lambda(Z) = \int \lambda'(Z) d\rho(\lambda'),$$
  we have $\lambda'(Z) > 0$ for each $\lambda'$ in some $\Lambda' \subset E(X^d)$ with $\rho(\Lambda') > 0$.
  Since each $\lambda' \in \Lambda'$ is ergodic and Z is an invariant subset of $X^d$, this implies $\lambda'(Z) = 1$ and, in particular, $\lambda'$ is separating.
\end{proof}

In later sections, we will show that a degree joining can be used to unpack all ergodic lifts of $\nu$ from it and that it is usually easier to construct a degree joining than to find all lifts of $\nu$ directly. But first, we show that degree joinings are unique up to permutations of the $d$ coordinates.

\begin{lemma}\label{lem:ergodic-rel-join-exists}
  Let $(X, T, \mu)$ and $(X', T', \mu')$ be two ergodic measure preserving systems with $(Y, S, \nu)$ as a common factor. Then there is an ergodic joining of the two systems over the common factor.
\end{lemma}
\begin{proof}
  We start by noting that there is at least one (not necessarily ergodic) joining $\lambda$ of $\mu, \mu'$ over $\nu$. In fact, it is easy to check that the relatively independent joining
  $$\lambda = \mu \otimes_\nu \mu' := \int \mu_y \otimes \mu'_y d\nu(y)$$
  is such a joining.

  It remains to show that the ergodic components of $\lambda$ satisfy the desired properties.
  Since $\mu, \mu', \nu$ are ergodic, almost all measures in the ergodic decomposition of $\lambda$ must also have $\mu, \mu'$ as their margins and $\nu$ as their image on $Y$. It is also easy to check that almost all measures in the ergodic decomposition are supported on the fiber product inside $X \times X'$.
\end{proof}

Degree joinings are universal with respect to other $\pi$-relative joinings over the same image in the following sense.
\begin{theorem}\label{thm:degree-joining-universal}
  Let $(X, Y, \pi, \nu)$ be a factor quadruple with finite degree $d$ and let $n$ be a positive integer.
  Let $\lambda$ be a degree joining over $\nu$ and $\lambda'$ an $n$-fold ergodic relative joining over $\nu$. Then there is a function $f: \{1,\dots, n\} \to \{1,\dots,d\}$ such that $\lambda' = p_f \lambda$ where $p_f: X^{d} \to X^n$ is the map induced by $f$ so that $$p_f(x_1, \dots, x_d) = (x_{f(1)}, \dots, x_{f(n)}).$$
\end{theorem}
We remark that in this theorem we do not assume $\lambda'$ to be separating. Therefore $n$ is allowed to be bigger than $d$ and $f$ does not have to be injective.

\begin{proof}
  There is an ergodic joining $\lambda''$ of $\lambda$ and $\lambda'$ over $\nu$, which follows from the previous lemma.
$\lambda''$ is a measure on $X^{d} \times X^{n}$.
Let $Y_0$ be a Borel subset of $Y$ such that for each $y \in Y_0$ the fiber $\pi^{-1}(y)$ consists of precisely $d$ points and $\nu(Y_0)=1$.
Let $Z_0$ be the set of all $(x_1, \cdots, x_d) \in X^d$ for which there is some $y \in Y_0$ such that $x_1, \cdots, x_d$ are the $d$ \emph{distinct} pre-images of $y$.

$Z_0$ is a Borel subset of full measure so that $\lambda(Z_0) = 1$ since it is the intersection of the following two sets, each of which is a Borel subset of $X^d$ of full measure:
\begin{align*}
  \{(x_1, \cdots, x_d) & : \pi(x_1) \in Y_0 \} \\
  \{(x_1, \cdots, x_d) & : \pi(x_1) = \cdots = \pi(x_d),\  x_i \ne x_j \text{ for all } 1 \le i < j \le d \}
\end{align*}

Let $W$ be the set of all $(x_1, \dots, x_d, x'_1, \dots, x'_n) \subset X^d \times X^n$ such that $\pi(x_1) = \cdots = \pi(x_d) = \pi(x'_1) = \cdots = \pi(x'_n)$ and $(x_1, \cdots, x_d) \in Z_0$. It is easy to see that $\lambda''(W) = 1$, because $\lambda''$ is a relative joining of $\lambda$ and $\lambda'$ over $\nu$ and $\lambda(Z_0)=1$.

For each $(x_1, \dots, x_d, x'_1, \dots, x'_n) \in W$, the points $x_1, \dots, x_d$ are $d$ distinct pre-images of a point $y$ in $Y_0$ (hence they are all the $d$ pre-images of that point $y$) and $(x'_1,\dots,x'_n)$ is a finite sequence of pre-images of the same point $y$, and therefore in particular, the point $x'_1$ for example is equal to one and only point among $x_1, \dots, x_d$. In other words,  there is a function $g: W \to \{1,\dots, d\}$ such that $$x'_1 = x_{g(x_1, \dots, x_d, x'_1, \dots, x'_n)}$$ holds for all $(x_1, \dots, x_d, x'_1, \dots, x'_n) \in W$.

The function $g$ is measurable and $\lambda''$-a.e. defined on $X^d \times X^n$. Since $g$ is $\lambda''$-a.e. invariant w.r.t. the product action $T \times T \times \cdots \times T$ on $X^d \times X^n$ and $\lambda''$ is ergodic, the function $g$ must be $\lambda''$-a.e. constant. Define $f(1)$ to be the a.e. constant value of $g$. Define $f(2), \dots, f(n)$ similarly. The function $f: \{1,\dots, n\} \to \{1,\dots,d\}$ defined in this way has the desired property because 
$$(x'_1,\dots, x'_n) = (x_{f(1)}, \dots, x_{f(n)}) = p_f(x_1,\dots,x_d)$$
holds for $\lambda''$-a.e. $(x_1, \dots, x_d, x'_1, \dots, x'_n)$.
\end{proof}
Conversely, each measure of the form $p_f \lambda$ where $\lambda$ is a degree joining over $\nu$ is an $n$-fold ergodic $\pi$-relative joining over $\nu$.
Since any relative joining over $\nu$ decomposes by ergodic decomposition into ergodic relative joinings over $\nu$, we have just classified all possible $\pi$-relative joinings over $\nu$ in the following sense. Any $n$-fold $\pi$-relative joining over $\nu$ is a convex combination $\sum_f a_f \cdot p_f \lambda$ for some coefficients $a_f \ge 0$ whose sum is 1, where $\lambda$ is a fixed degree joining. This is a finite convex combination because there are only $d^n$ possibilities for $f$.

Universality implies uniqueness of degree joining up to permutation as proved in the following theorem.
\begin{theorem}\label{thm:degree-joining-unique}
  Let $(X, Y, \pi, \nu)$ be a factor quadruple with finite degree $d$ and let $n$ be a positive integer. If $\lambda$ and $\lambda'$ are degree joinings over $\nu$, then there is a permutation $f$ of $\{1, \dots, d\}$ such that $\lambda' = p_f\lambda$ and therefore also $\lambda = p_{f^{-1}}\lambda'$.
\end{theorem}
\begin{proof}
  There is a function $f: \{1,\dots, d\} \to \{1,\dots,d\}$ such that $\lambda' = p_f \lambda$.
  Suppose to the contrary that $f$ is not surjective. Without loss of generality, we may assume $f(1) = f(2) = 1$.

  For $\lambda$-a.e. $(x_1, \dots, x_d)$ we have that $p_f(x_1, \dots, x_d)$ is of the form $(x'_1, \dots, x'_d)$ with $x'_1 = x'_2$. Therefore, for $\lambda'$-a.e. $(x'_1, \dots, x'_d)$, we have $x'_1 = x'_2$ but this contradicts the assumption that $\lambda'$ is separating.
\end{proof}

\section{Multiplicity structure}\label{sec:mult}
In this section, we define multiplicity of ergodic measures on $X$ and extract the multiplicity structure of the measure fiber from the degree joining.

Having established the uniqueness of degree joining, we now show that its margins are precisely the ergodic lifts of $\nu$.
This property is why degree joinings are a useful tool to investigate the lifts of ergodic measures under a.e. finite-to-one factor maps.

\begin{theorem}
  Let $(X, Y, \pi, \nu)$ be a factor quadruple with finite degree $d$ and $\lambda$ a degree joining over $\nu$. Then $$\{p_i \lambda: 1\le i \le d\}$$ is the set of all ergodic measures in $\pi^{-1}(\nu)$.
\end{theorem}
\begin{proof}
  Each margin $p_i\lambda$ is an ergodic measure on $X$ that maps to $\nu$, because $\lambda$ is an ergodic joining over $\nu$ and each projection $p_i: X^d \to X$ is a factor map.

  Each ergodic measure in $\pi^{-1}\nu$ is a 1-fold ergodic relative joining over $\nu$ and hence Theorem~\ref{thm:degree-joining-universal} applies to it and therefore is one of the margins of $\lambda$.
\end{proof}

The above theorem allows us to define multiplicity of ergodic measures in the following way.

\begin{definition}
  Let $(X, Y, \pi, \nu)$ be a factor quadruple with finite degree $d$. Let $\mu \in E(X)$ be an ergodic lift of $\nu$. The \emph{multiplicity}, denoted $m_\pi(\mu)$, of $\mu$ with respect to $\pi$ is the number of times it appears as a margin in a degree joining over $\nu$. In other words,
  $$m_\pi(\mu) := \#\{i : 1\le i \le d,\ p_i \lambda = \mu\}$$
  where $\lambda$ is a degree joining over $\nu$.
\end{definition}
Since degree joining is unique up to permutation, the notion of multiplicity above is well defined, i.e., it does not depend on the choice of $\lambda$. Our goal in defining this notion was to establish the following result which looks like a trivial result until we give different characterizations of multiplicity later in this section.
\begin{theorem}
  Let $(X, Y, \pi, \nu)$ be a factor quadruple with finite degree $d$.
  Then
  $$d = \sum_\mu m_\pi(\mu)$$
  where $\mu$ runs over all ergodic lifts of $\nu$.
\end{theorem}
In particular, the degree $d$ is an upper bound on the number of ergodic lifts. In some sense, we can say there are $d$ such lifts if we count with multiplicity.

\begin{theorem}\label{thm:fiber-generic-points}
  Let $(X, Y, \pi, \nu)$ be a factor quadruple with finite degree $d$. Then for $\nu$-a.e. $y \in Y$, each point in the fiber $\pi^{-1}(y)$ is a generic point for some ergodic measure in $\pi^{-1}(\nu)$. Furthermore, let $\mu_1, \dots, \mu_k$ be all ergodic lifts of $\nu$ and let $m_1,\dots, m_k$ be their multiplicities. Then for $\nu$-a.e. $y \in Y$, the fiber $\pi^{-1}(y)$ consists precisely of $m_1$ points generic for $\mu_1$, and $m_2$ points generic for $\mu_2$, \dots, and $m_k$ points generic for $\mu_k$.
\end{theorem}
\begin{proof}
  Let $\lambda$ be a degree joining over $\nu$. For $\lambda$-a.e. $(x_1, \dots, x_d)$, we have that $x_1$ is generic for $p_1 \lambda$, and $x_2$ is generic for $p_2 \lambda$, and so on. The desired conclusion follows by transferring to $Y$.

More precisely, for each $1 \le i \le d$, since $\mu_i$ is ergodic, there is a Borel subset $G_i \subset X$ such that $\mu_i(G_i)=1$ and every $x \in G_i$ is generic for $\mu_i$. Let $Z_0 \subset X^d$ be defined as in the proof of Theorem~\ref{thm:degree-joining-universal}. Then the intersection
$$ Z' := Z_0 \cap  (p_1^{-1} G_1 \cap \cdots \cap p_d^{-1} G_d)$$
is a Borel subset of $X^d$ of full measure. Its image $Y' := \pi p_1 Z'$ in $Y$ is a ($\nu$-measurable but not necessarily Borel) full measure set too, i.e., $\nu(Y')=1$ because it is the image of a full measure set under the measure preserving map $\pi p_1: X^d \to Y$.
By definition, every $y \in Y'$ satisfies the desired properties.
\end{proof}
The above theorem implies in particular that if we are given a factor quadruple $(X, Y, \pi, \nu)$ with finite degree, we can read off the preimage measures $\mu_1, \dots, \mu_k$ and their multiplicities by just looking at the set $\pi^{-1}(y)$ after fixing a random point $y \in Y$ chosen according to $\nu$. Note that within the class of finite-to-one factor codes on SFTs, there are varying levels of difficulty in extracting $\pi^{-1}(y)$ from $y$ depending on $\pi$. For example, the easiest case for reading off $\pi^{-1}(y)$ from $y$ is the class of bi-closing factor codes and the next easiest case is the class of right-closing factor codes. See \cite{LM} for definitions and properties of such classes.

\begin{theorem}
  Let $(X, Y, \pi, \nu)$ be a factor quadruple with finite degree $d$. Let $\mu \in E(X)$ be an ergodic lift of $\nu$ and let $m$ be its multiplicity. Let $\{\mu_y\}_{y \in Y}$ be the disintegration of $\mu$ over $Y$. Then
  \begin{enumerate}
  \item For $\nu$-a.e. $y$, the measure $\mu_y$ is uniformly distributed on $G_\mu \cap \pi^{-1}(y)$, where $G_\mu$ is the set of points generic for $\mu$, and there are exactly $m$ points in $G_\mu \cap \pi^{-1}(y)$.
  \item $m$ is the maximum number such that there is an $m$-fold separating relative joining of margins $\mu, \dots, \mu$ over $\nu$.
  \item $(\mu \otimes_{\nu} \mu) \{(x, x'): x=x'\} = \frac1m$
  \item The factor map $\pi: (X, \mu, T) \to (Y, \nu, S)$ seen as a homomorphism between two ergodic systems is (a.e.) $m$-to-one up to null set in the sense that almost every $\mu_y$ is an atomic measure that gives positive measure to exactly $m$ points.
  \end{enumerate}
\end{theorem}
\begin{proof}

(1)
To prove the first property, we start by observing that $G_\mu$ is a $T$-invariant Borel subset of $X$ such that $\mu(G_\mu)=1$ and by Theorem~\ref{thm:fiber-generic-points} that the canonical lift $\ell_\pi(\nu)$ satisfies
$$\ell_\pi(\nu)(G_\mu) = \frac{m}{d} > 0.$$
Therefore, the conditional measure $\mu'$, resulting from conditioning the canonical lift to $G_\mu$, defined by 
$$\mu'(A) := \frac{\ell_\pi(\nu)(A \cap G_\mu)}{\ell_\pi(\nu)(G_\mu)} = \frac{d}{m} \cdot {\ell_\pi(\nu)(A \cap G_\mu)}$$
for each Borel $A \subset X$, is an invariant probability measure on $X$.

The image of $\mu'$ on $Y$ is $\nu$ because otherwise $\nu$ can be written as a nontrivial convex combination of two invariant measures $\pi\mu'$ and $\pi\mu''$ both different from $\nu$ where $\mu''$ is the conditional measure resulting from conditioning the canonical lift to the complement of $G_\mu$ and that would contradict the ergodicity of $\nu$.
Now it is straightforward to show that the disintegration of $\mu'$ over $Y$ satisfies the property that for $\nu= \pi\mu'$-a.e. $y \in Y$, the measure $\mu'_y$ is the uniform distribution on the $m$-points set $G_\mu \cap \pi^{-1}(y)$.

To prove the first property in the theorem for $\mu$, it only remains to show that $\mu=\mu'$, but that follows from $\mu'(G_\mu)=1$ and the fact that the only invariant measure supported on $G_\mu$ is $\mu$ itself.

(2) 
We can obtain an $m$-fold separating relative joining of $\mu, \dots, \mu$ over $\nu$ by projecting a degree joining over $\nu$ to the $m$ coordinates for which $\mu$ is the corresponding margin.
Now suppose $\lambda'$ is an $(m+1)$-fold such joining. Then for $\lambda'$-a.e. $(x_1,\dots, x_{m+1})$, the points $x_1,\dots, x_{m+1}$ are $m+1$ distinct points and they are all in $G_\mu \cap \pi^{-1}(\pi (x_1))$. By transferring this observation to $(Y, \nu)$, we have that for $\nu$-a.e. $y$, the size of $G_\mu \cap \pi^{-1}(y)$ is at least $m+1$. But this contradicts Theorem~\ref{thm:fiber-generic-points} and so there can be no such $(m+1)$-fold joining.

(3) For $\nu$-a.e. $y$, we have $(\mu_y \otimes \mu_y) \{(x, x'): x=x'\} = \frac1m$ because $\mu_y$ is the uniform distribution on $m$ points. Integrating over $(Y, \nu)$ gives the desired result.

(4) This follows from (1).
\end{proof}

\begin{remark}
  Each of the four properties shown in the above theorem can be taken to be an alternative (but equivalent) characterization/definition of the multiplicity of $\mu$. The second property can be interpreted as saying that for the measure $\mu$ to have multiplicity bigger than one, the ($\nu$-almost every) fiber $\pi^{-1}(y)$ must allow some room for a copy of $\mu$ to get in to form a 2-fold separating self-joining. Each of the last two properties characterizes $m$ as something that depends only on the isomorphism mod 0 class of the corresponding homomorphism $\pi: (X, \mu, T) \to (Y, \nu, S)$ between the induced ergodic measure preserving transformations. In particular, $m$ is just the size of the fiber component resulting from applying the Rohlin's skew-product theorem to that homomorphism.
\end{remark}

Using the notion of canonical lift, we can obtain yet another characterization of the notion of multiplicity, as weights in the ergodic decomposition of the canonical lift.
\begin{theorem}
 Let $(X, Y, \pi, \nu)$ be a factor quadruple with finite degree $d$. Let $\mu_1, \dots, \mu_k$ be all ergodic lifts of $\nu$ and let $m_1,\dots, m_k$ be their multiplicities. Then the ergodic decomposition of the canonical lift of $\nu$ is given by
 $$\ell_\pi(\nu) = \sum_{i=1}^{k} \frac{m_i}{d} \cdot \mu_i$$
\end{theorem}
\begin{proof}
For each $1 \le i \le k$, let $G_i \subset X$ be the set of all points generic for $\mu_i$.
In the proof of the first property in the previous theorem, we showed that 
$$ \ell_\pi(\nu)(G_i) = \frac{m_i}{d} $$
and that
$$ \ell_\pi(\nu)(\cdot | G_i) = \mu_i(\cdot) $$

Since $ \ell_\pi(\nu)(\cup_{i=1}^k G_i) = \sum_{i=1}^k \frac{m_i}{d} = 1 $, the collection $\{G_i\}_{1 \le i \le k}$ forms a mod 0 partition of the probability space $(X, \ell_\pi(\nu))$. We disintegrate the probability space w.r.t. this partition to obtain

$$ \ell_\pi(\nu) = \sum_{i=1}^{k} \ell_\pi(\nu)(G_i) \ell_\pi(\nu)(\cdot | G_i) = \sum_{i=1}^{k} \frac{m_i}{d} \mu_i $$

Since the measures $\mu_i$ are distinct ergodic measures and the coefficients $\frac{m_i}{d}$ are positive, the above decomposition is also the ergodic decomposition.
\end{proof}

In the above sense, the canonical lift contains all possible ergodic lifts of $\nu$.

\begin{corollary}
  Let $(X, Y, \pi, \nu)$ be a factor quadruple with finite degree $d$. The canonical lift $\ell_\pi(\nu)$ is ergodic if and only if there is only one invariant measure $\mu$ on $X$ that projects to $\nu$, in which case the canonical lift is that one measure $\mu$.
\end{corollary}
\begin{proof}
  Since $\nu$ is ergodic, almost every ergodic component of any invariant measure on $X$ that projects to $\nu$ is again a measure that projects to $\nu$. Therefore, there is only one invariant measure on $X$ that projects to $\nu$ if and only if there is only one ergodic lift of $\nu$. By the previous theorem, this is the case if and ony if the canonical lift is itself ergodic.
\end{proof}

\section{Examples}\label{sec:ex}
In this section, we build some examples before we move onto a general theory of degree joinings for symbolic dynamics.

Examples in this section come from endomorphisms of full shifts, except for one example. In particular, $X$ and $Y$ are always the same full shift in the examples. In terms of cellular automata theory, examples here are based on two linear cellular automata that generalize the rule 102 automaton in Example~\ref{first-example}. The two endomorphisms we introduce share the special property that $|\pi^{-1}(y)|$ do not depend on $y \in Y$, in other words, they are constant-to-one factor codes.
Recall that in this case, the degree $d_\nu$ of the measure $\nu\in E(Y)$ is the same for all $\nu\in E(Y)$ including those $\nu$ that are not fully supported.
We will see that measure fibers already exhibit diverse behavior within this simple class of factor codes.

\begin{example}
Let $N \in \mathbb N$.  Let $X = Y$ be the full $N$ shift. Then the factor code $\pi:X \to Y$ defined by
$$x = (x_i)_{i\in\mathbb Z} \mapsto \pi(x) := (x_{i+1}-x_{i})_{i\in\mathbb Z} \pmod{N}$$
is an $N$-to-1 map. Indeed, if the map $s: X\to X$ is defined by
$$x = (x_i)_i \mapsto (x_i+1)_i \pmod{N},$$
then we have
$$\pi^{-1}\pi x = \{x, s(x),\dots, s^{N-1}(x) \} = \{s^k(x) : k \in \mathbb Z \},$$
for all $x \in X$. For any ergodic $\mu\in E(X)$, its image $\lambda$ under the map
$$x \mapsto (x, s(x),\dots, s^{N-1}(x))$$
is a degree joining over $\pi\mu$. ($\lambda$ is ergodic because it is an image of $\mu$ under a shift-commuting map.)
Therefore, all ergodic lifts of $\pi\mu$ are in the list $\mu, s(\mu), \dots, s^{N-1}(\mu)$ and the multiplicity of $\mu$ is the number of times it appears in the list and is therefore always a divisor of $N$. The number of ergodic lifts of $\pi\mu$ also divides $N$ and $N$ is the product of that number and the multiplicity of $\mu$.

In particular, if $\mu$ is the Bernoulli product measure on $X$ given by a probability vector $(\alpha_1,\dots, \alpha_N)$, then its multiplicity is $\frac{N}{L}$ where $L$ is the least period of the sequence $(\alpha_1,\dots, \alpha_N)$.
For almost all probability vector $(\alpha_1,\dots, \alpha_N)$, the value of $L$ is the full length $N$ and $\mu, s(\mu), \cdots, s^{N-1}(\mu)$ are $N$ different lifts of $\pi\mu$.
If $N=4$ and $(\alpha_1,\dots, \alpha_4) = (\frac18, \frac38, \frac18, \frac38)$, then we have $$(\mu, s(\mu), s^2(\mu), s^3(\mu)) = (\mu, s\mu, \mu, s\mu)$$ where $s\mu$ is the different Bernoulli product measure from the shifted vector $(\frac38, \frac18, \frac38, \frac18)$. In this case, the number of ergodic lifts of $\pi\mu$ is 2, which is strictly between 1 and the degree $4$. This is over a fully supported member in $E(Y)$ and therefore should be considered less trivial than CO-measures (ergodic measures supported on periodic orbits) in $E(Y)$.

On the other hand, finding and verifying an example of $\nu \in E(Y)$ with precisely two ergodic lifts on $X$ for $N=4$ without requiring full support is elementary and does not require the degree joining theory. For example, the CO-measure $\nu$ supported on the fixed point $\cdots 222.222 \cdots \in Y$ has four pre-images in $X$ who form two periodic orbits, each with least period 2. One of the two periodic orbits is the orbit of $\cdots 0202.0202 \cdots \in X$ and the other is the orbit of $\cdots 1313.1313\cdots \in X$. The two CO-measures supported on these two periodic orbits are precisely the ergodic lifts of $\nu$ and they have multiplicity 2 because the map $\pi$ collapses each of these periodic orbits by halving their periods. We generalize this observation about CO-measures shortly in the next example before moving to a more complicated example.
\end{example}

Given the example, we raise the following question.
\begin{question}
  With $\pi: X\to Y$ from the previous example, is there a subset $E' \subset E(Y)$ such that $E'$ is a residual set in the simplex of all invariant measures on $Y$ and that each $\nu \in E'$ has exactly $N$ ergodic lifts?
\end{question}
We suspect the answer is yes.
We remark that the fully supported ergodic measures form a residual set in the simplex \cite{DGS1976ergodic}. Therefore the question can also be thought of as a question on the residual behavior of measure fibers over fully supported ergodic measures on $Y$.

\begin{example}
  \label{ex:periodic}
  Let $\pi: (X, T) \to (Y, S)$ be a factor map between topological dynamical systems. Let $\nu$ be the CO-measure supported on some periodic orbit $Y'$ in $Y$ with least period $p$. Then $d_\nu$ is finite iff the fiber over some point (or equivalently, every point) in the periodic orbit is finite. Now suppose $d_\nu$ is finite and let $d=d_\nu$. The inverse image $X' = \pi^{-1}(Y')$ contains $dp$ points. Since $X'$ is a finite subsystem of $X$, it consists only of periodic points and hence is some finite disjoint union of periodic orbits in $X$. Recall that CO-measures can be identified with periodic orbits. It is easy to verify that the ergodic lifts of $\nu$ are exactly the periodic orbits in $X'$. For each periodic orbit $X'_i$ in $X'$, let $m_i = \frac{|X'_i|}{|Y'|} = \frac{|X'_i|}{p}$. The number $m_i$ measures how $\pi$ folds the periodic orbit $X'_i$ and can be thought of as a discrete winding number. A simple counting argument shows that $\sum_i m_i = d$. It is also easy to verify that the winding number $m_i$ is the multiplicity of the CO-measure $\mu_i \in E(X)$ corresponding to $X'_i$ since the map $\pi: (X, \mu_i) \to (Y, \nu)$ seen as a factor map between two ergodic systems is a.e. $m_i$-to-one. In this sense, the multiplicity of an arbitrary ergodic measure generalizes the winding number of periodic orbits. Therefore, another motivation for the multiplicity theory is the viewpoint that the ergodic measures on a topological dynamical system is a generalization of periodic orbits.
\end{example}

For the next example which exhibits a more complicated behavior for the measure fibers, we need a lemma.

\begin{lemma}
  Let  $(X, T, \mu)$ be an ergodic system. Denote by $\mathbbm 2 = (2, S, \nu)$ the unique ergodic system consisting of two atoms. Then the following are equivalent.
  \begin{enumerate}
  \item The system $(2,S,\nu)$ is a factor of $(X,T,\mu)$.
  \item The product system $(X \times 2, T\times S, \mu \otimes \nu)$ is not ergodic.
  \end{enumerate}
  If these conditions hold, we will say that $\mathbbm 2$ is a factor of $\mu$.
\end{lemma}
\begin{proof}
  If $(2, S, \nu)$ is a factor of $(X, T ,\mu)$, then $(X, T, \mu) \times (2, S, \nu)$ has a factor $(2, S, \nu) \times (2, S, \nu)$ which is not ergodic and therefore the product system is not ergodic.

  It remains to show $\neg (1) \implies \neg (2)$. Let $f: X \times 2 \to \mathbb R$ be a $\mu \otimes \nu$-a.e. $T\times S$-invariant measurable function. We want to show that this function is a.e. constant. Since $f$ is invariant, $f(x, 0) = f(Tx, 1)$ and $f(x, 1) = f(Tx, 0)$ hold for a.e. $x$. So $f(x, 0) + f(x, 1)$ is $T$-invariant and hence, by the ergodicity of $T$, a.e. constant. So for some $r \in \mathbb R$ we have $f(x, 0) + f(x, 1) = r$ a.e.
On the other hand, we have $f(x, 0) - f(x, 1) = - (f(Tx, 0) - f(Tx, 1))$. So $f(x, 0) - f(x, 1)$ is a.e. zero, because otherwise it would be a.e. nonzero by the ergodicity of $T$ and then the sign of $f(x, 0) - f(x, 1)$ can be used to form a factor map to $(2, S, \nu)$ which would contradict our starting assumption. So $f(x, 0) = f(x, 1)$ holds a.e. and therefore $f(x, 0) = f(x, 1) = \frac{r}2$.
\end{proof}

\begin{example}
  Let $X=Y$ be the full 5 shift. Then the factor code $\pi:X \to Y$ defined by
$$x = (x_i)_i \mapsto (x_{i+1}+x_{i})_i \pmod{5}$$
is a 5-to-1 map. Unlike the previous example, we are taking the sum of two consecutive numbers instead of taking the difference, making it impossible to define a shift-commuting function $s: X\to X$ to sweep inside fibers as before. Let $\mu$ be an ergodic measure on $X$ such that $\mathbbm 2$ is not a factor. To form a degree joining over $\pi\mu$, we need some auxiliary measure. Let $\eta$ be the unique ergodic measure on the shift space $Z$ consisting of two points $((-1)^{i})_i$ and $((-1)^{i+1})_i$. The image $\lambda$ of $\mu \otimes \eta$ under the map
$$(x, z) \mapsto (x, x+z, x+2z, x+3z, x+4z) \mod5$$
is a degree joining over $\pi\mu$. ($\lambda$ is ergodic because $\mu\otimes\eta$ is, by the previous lemma.) Since $x+4z \equiv x-z \mod5$ and the image of $\eta$ under the map $z \mapsto -z$ is $\eta$, we can verify that the second margin and the last margin of $\lambda$ are the same and we denote it by $\mu'$. Also, the third margin and the fourth margin are the same and we denote it by $\mu''$. The measures $\mu, \mu', \mu''$ are all ergodic lifts of $\pi\mu$.
In many cases of $\mu \in E(X)$, $\mu, \mu', \mu''$ are three distinct measures. One such case will be mentioned in the next theorem. When they are distinct, their multiplicities are 1, 2, 2 respectively and the number of ergodic lifts of $\pi\mu$ is strictly between 1 and the degree 5 and does not divide the degree.
Whenever $\nu$ is an ergodic measure on $Y$ such that $\mathbbm 2$ is not a factor, the number of its ergodic lifts is at most 3.
\end{example}

\begin{theorem}
  Let $\pi: X\to Y$ be the factor code from the previous example. Let $P = \{x\in X: x_0 = 0\}$. Let $\mu \in E(X)$ be such that $\mathbbm 2$ is not a factor and $\mu(P)>\frac12$. Then $\pi\mu$ has exactly three ergodic lifts on $X$.
\end{theorem}
\begin{proof}
  Let $\mu', \mu''$ be as in the previous example. It is enough to show that $\mu, \mu', \mu''$ are different.
  Let $P' = \{x\in X: x_0 \in \{1, 4\}\}$ and $P'' = \{x\in X: x_0 \in \{2, 3\}\}$. Then we have
  \begin{align*}
    \mu'(P') &= (\mu \otimes \eta)(\{(x,z): x+z \in P' \pmod5 \}) \\
    &= (\mu \otimes \eta)(\{(x,z): x_0+z_0 \in \{1, 4\} \pmod5 \}) \\
    &\ge (\mu \otimes \eta)(\{(x,z): x_0 = 0, z_0 \in \{\pm 1\}\}) \\
    &= \mu(P) > \frac12
  \end{align*}
  Similarly, $\mu''(P'') > \frac12$.
  Since $P, P', P''$ are disjoint subsets of $X$ such that $\mu(P)>\frac12, \mu'(P')>\frac12, \mu''(P'')>\frac12$, the three measures must be different. 
\end{proof}
In particular, if $\mu$ is the Bernoulli product measure from any probability vector $(\alpha_0, \alpha_1, \cdots, \alpha_4)$ with $\alpha_0 > \frac12$, it satisfies the hypothesis of the theorem.
This example demonstrates that it is possible to have different multiplicities within one measure fiber over a fully supported ergodic measure.

\begin{example}
  Let $\pi: X\to Y$ and $(Z, \eta)$ be from the previous example, but this time we suppose $\mu\in E(X)$ has $\mathbbm 2$ as a factor. There is a shift-commuting measurable function $F: X \to Z$ such that $F\mu = \eta$. The image of $\mu$ under the map
$$x \mapsto (x, x+ F(x), x+ 2F(x), x+ 3F(x), x+ 4F(x)) \mod5$$
is a degree joining over $\pi\mu$.
When the margins are different, their multiplicities will be 1.
\end{example}

\begin{question}
  With $\pi: X\to Y$ from the previous example, is there a subset $E' \subset E(Y)$ such that $E'$ is a residual subset of the simplex of invariant measures on $Y$ and that each $\nu \in E'$ has exactly 3 ergodic lifts?
\end{question}

\begin{remark}
  One can similarly investigate the factor code with $5$ replaced by arbitrary $N>1$. We only did $N=5$ because $5$ was the smallest number to reveal the general pattern for larger $N$.
\end{remark}

\section{Degree joinings for finite to one factor codes}\label{sec:deg-symb}

In this section, we identify degree joinings for general finite to one factor codes. In this case, we show that a degree joining can be obtained by just lifting $\nu$ to an easily constructed subshift of finite type, which we will call topological degree joining.
First we recall some facts from the classical theory of degree of such factor codes.

\begin{theorem}[Theorem  8.1.19 in~\cite{LM}]\label{thm:fto-conditions}
  Let $X$ be an irreducible sofic shift and $\pi: X\to Y$ a factor code (hence $Y$ is also an irreducible sofic shift). Then the following are equivalent. The factor codes satisfying any of these conditions are called \emph{finite-to-one} factor codes.
  \begin{enumerate}
  \item For every $y \in Y$, the fiber $\pi^{-1}(y)$ is countable.
  \item For every $y \in Y$, the fiber $\pi^{-1}(y)$ is finite.
  \item The map is bounded-to-one, i.e., there is $M \in \mathbb N$ such that, for every $y \in Y$, $|\pi^{-1}(y)| \le M$.
  \item $X$ is a relative zero entropy extension of $Y$, i.e., $h(X) = h(Y)$.
  \end{enumerate}
\end{theorem}

\begin{theorem}[Lemma 9.1.13 in~\cite{LM}]\label{thm:fto-and-transitive}
  Let $X$ be an irreducible sofic shift and $\pi: X\to Y$ a finite-to-one factor code. Then a point $x \in X$ is doubly transitive if and only if its image $\pi(x)$ is.
\end{theorem}

\begin{theorem}[Corollary  9.1.14 in \cite{LM}]\label{thm:fto-has-degree}
  Let $X$ be an irreducible sofic shift and $\pi: X\to Y$ a finite-to-one factor code. There is $d_\pi \in \mathbb N$ such that each doubly transitive point in $Y$ has exactly $d_\pi$ pre-images. This number $d_\pi$ is called the \emph{degree} of the factor code $\pi$.
\end{theorem}

A set of points in a 1-step SFT is \emph{mutually separated} if each pair of points never occupy the same symbol at the same time.

\begin{theorem}[\cite{LM}]\label{thm:fto-mut-sep}
  Let $X$ be an irreducible 1-step SFT and $\pi: X\to Y$ a finite-to-one 1-block factor code. Let $y \in Y$ (not necessarily doubly transitive). Then there are at least $d_\pi$ mutually separated points in the fiber $\pi^{-1}(y)$. In particular, if $y\in Y$ is such that $\pi^{-1}(y)=d_\pi$ (which is the case whenever $y$ is doubly transitive or $\pi$ is constant-to-one), then all points in the fiber are mutually separated.
\end{theorem}
To prove the above well known theorem using only propositions in \cite{LM}, one can follow Proposition 9.1.9 in it to establish the theorem for doubly transitive $y\in Y$ first, and then pass to arbitrary $y \in Y$ by a diagonal argument. The diagonal argument works because the property of being mutually separated is preserved under limits.

For the symbolic dynamical case of this section, we will mainly work with those $\nu \in E(Y)$ that are fully supported, because $d_\nu = d_\pi$ holds for all such $\nu$. This follows from the following lemma.

\begin{lemma}\label{lem:dt-full-in-full}
  Let $(X, T)$ be a topological dynamical system and let $\mu\in E(X, T)$ be fully supported. Then the set of doubly transitive points is a full measure set w.r.t. $\mu$.
\end{lemma}
\begin{proof}
  Let $U$ be a non-empty open set. Since $\mu$ is ergodic, $\mu$-a.e. $x\in X$ has the property that its forward orbit visits $U$ with frequency given by $\mu(U)$. But $\mu(U)$ is positive because $\mu$ has full support. Therefore $\mu$-a.e. $x$ is forward transitive. Using the inverse map $T^{-1}$, it follows that $\mu$-a.e. $x$ is backward transitive as well.
\end{proof}

Another reason we work with fully supported measures is that they are preserved under lifting via finite-to-one factor codes. This is a measure-theoretical analogue of Theorem~\label{thm:fto-and-transitive} and is a direct consequence of the following lemma.
\begin{lemma}
  Let $(X, T)$ be a topological dynamical system that is entropy minimal, i.e., every proper subsystem of $(X, T)$ has strictly smaller entropy. Let $\pi: (X, T) \to (Y, S)$ be a factor map and $h(Y, S)=h(X,T)$. Then an invariant measure $\mu$ on $X$ has full support if and only if $\pi\mu$ has full support.
\end{lemma}
\begin{proof}
  Since $\pi$ is continuous, we have $\pi(\supp(\mu)) = \supp(\pi\mu)$. In particular, $\supp(\mu)=X$ implies $\supp(\pi\mu)=Y$.
  Now it remains to prove the converse. Suppose $\pi\mu$ has full support but $\mu$ does not. Let $X_0 = \supp(\mu)$. Then $X_0$ is a proper subsystem of $X$, hence $h(X) > h(X_0)$, but we also have $\pi(X_0) = \supp(\pi\mu) = Y$ and hence $h(X_0) \ge h(\pi(X_0)) = h(Y)$. Therefore, $h(X) > h(X_0) \ge h(Y)$ which contradicts the equal entropy assumption $h(X)=h(Y)$.
\end{proof}

\begin{lemma}\label{lem:fto-full}
  Let $X$ be an irreducible sofic shift and $\pi: X\to Y$ a finite-to-one factor code. An invariant measure $\mu$ on $X$ has full support if and only if the pushforward image $\pi\mu$ has full support.
\end{lemma}
\begin{proof}
  An irreducible sofic shift is entropy minimal (see~\cite{LM}). $h(X)=h(Y)$ follows from Theorem~\ref{thm:fto-conditions}. Therefore, the previous lemma applies in this case.
\end{proof}

Let $\pi: (X, T) \to (Y,S)$ be a factor map between topological dynamical systems. For each $n >0$, recall that the $n$-fold fiber product $X^n_\pi$ is a subsystem of $X^n$ where $n$-fold relative joinings live.
It is easy to check that $X^n_\pi$ is an $n$-fold \emph{topological joining} of $X$ with itself, in other words, it is a subsystem of $X^n$ with projections $p_i(X^n_\pi) = X$ for each $1\le i \le n$.
In the symbolic case where $\pi$ is 1-block factor code on a 1-step SFT $X$, the fiber product $X^n_\pi$ is also a 1-step SFT.

\begin{definition}\label{def:top-deg-join}
  Let $\pi: X \to Y$ be a finite-to-one factor code from an irreducible SFT $X$ and let $d$ be the degree of $\pi$. Additionally, we assume $\pi$ is recoded, in other words, we assume that $\pi$ is a 1-block factor code and $X$ is a 1-step SFT. The \emph{topological degree joining} for the code $\pi$ is the set $\Lambda$ of all $(x^{(1)}, x^{(2)}, \dots, x^{(d)}) \in X^d_\pi$ such that $x^{(1)}, x^{(2)}, \dots, x^{(d)}$ are $d$ distinct mutually separated points.
\end{definition}
It is easy to check that the topological degree joining $\Lambda$ is a 1-block SFT contained in the self fiber product $X^d_\pi$, but it is in general not irreducible.
\begin{theorem}
The $d$ projections of $\Lambda$ are all $X$, that is, $p_i(\Lambda)=X$ for each $1\le i\le d$. In particular, $\Lambda$ is a $d$-fold topological joining of $X$.
The map $\pi_\Lambda: \Lambda \to Y$ defined by $\pi_\Lambda = \pi \circ p_1 = \pi \circ p_2 = \cdots = \pi \circ p_d$ is a 1-block factor code. In particular, $\pi_\Lambda$ is a finite-to-one 1-block factor code from a (not necessarily irreducible) 1-step SFT.
\end{theorem}
\begin{proof}
Theorem~\ref{thm:fto-mut-sep} implies that the map $\pi_\Lambda$ is onto. It is a 1-block factor code because $\pi$ and $p_i$ are. It is finite-to-one because it is a restriction of the finite-to-one map $X^d_\pi \to Y$.

In order to show $p_i(\Lambda)=X$, fix $i$ and let $X_0 := p_i(\Lambda) \subset X$. Since $\pi_\Lambda = \pi \circ p_i |_\Lambda$ is onto, it follows that the subsystem $X_0$ projects onto $Y$ under $\pi$. Now we use the same argument as in the proof of Lemma~\ref{lem:fto-full}.
We have $h(X_0) \ge h(\pi(X_0)) = h(Y) = h(X)$. By entropy minimality of $X$, this implies $X_0 = X$, in other words, $p_i(\Lambda)=X$.
\end{proof}

We can now show that $\Lambda$ is a space hosting all degree joinings over all possible $\nu$ with full support.

\begin{theorem}\label{thm:for-sft}
Under the assumptions from Definition~\ref{def:top-deg-join}, for each fully supported $\nu\in E(Y)$, the quadruple $(X,Y,\pi, \nu)$ is a factor quadruple with degree $d$. In this case, the set of all degree joinings over $\nu$ w.r.t. $\pi$ is
\[\{\lambda\in E(\Lambda): \pi_\Lambda \lambda = \nu\}.\] 
\end{theorem}
\begin{proof}
Since $\Lambda \subset X^d_\pi$, each invariant measure $\lambda$ on $\Lambda$ is a $d$-fold $\pi$-relative joining. Such $\lambda$ is separating because mutually separated points are distinct. Therefore, each member of the set $\{\lambda\in E(\Lambda): \pi_\Lambda \lambda = \nu\}$ is a degree joining over $\nu$.

Conversely, suppose $\lambda$ is a degree joining over $\nu$.

First we show that $\lambda$-a.e. $(x^{(1)}, \cdots, x^{(d)})$ is mutually separated. By definition, for $\lambda$-a.e. $(x^{(1)}, \cdots, x^{(d)})$, the points $x^{(1)}, \cdots, x^{(d)}$ are the $d$ distinct pre-images of $\pi (x^{(1)})$. Since $\nu$-a.e. $y$ is doubly transitive and $\nu = \pi p_1 \lambda$, we can conclude that the point $\pi (x^{(1)})$ is doubly transitive for $\lambda$-a.e. $(x^{(1)}, \cdots, x^{(d)})$. Therefore its $d$ pre-images are mutually separated by Theorem~\ref{thm:fto-mut-sep}. This shows that $\lambda$-a.e. $(x^{(1)}, \cdots, x^{(d)})$ is indeed mutually separated and so $\lambda(\Lambda)=1$.
Now, it follows easily that $\lambda$ is in $E(\Lambda)$ with $\pi_\Lambda \lambda = \nu$ by definition of $\pi_\Lambda$.
\end{proof}

The theorem above implies that in order to construct all lifts of $\nu$ through the factor code $\pi: X\to Y$, it is enough to lift $\nu$ to an ergodic measure on $\Lambda$ through the new factor code $\pi_\Lambda$ just once and then obtain all lifts in $X$ as margins of the constructed degree joining. This also works as a more constructive proof of existence of a degree joining for the symbolic dynamics case because the non-emptiness of the set $\{\lambda\in E(\Lambda): \pi_\Lambda \lambda = \nu\}$ follows directly from the fact that $\pi_\Lambda$ is a factor map onto $Y$.
We note that the SFT $\Lambda$ is easily computable from the code $\pi: X\to Y$ in the following precise sense. Using the construction of labeled products of labeled graphs (see~\cite{LM}), one can represent the topological degree joining $\Lambda$ together with $\pi_\Lambda$ as a subgraph of the labeled product of $d$ copies of the labeled graph representing $\pi$.

If the factor code $\pi$ is constant-to-one, then we obtain the same result for all $\nu\in E(Y)$ even when $\nu$ is not fully supported:
\begin{theorem}
In addition to the assumptions from Definition~\ref{def:top-deg-join}, also assume that $\pi$ is constant-to-one, i.e., each $y\in Y$ has precisely $d$ pre-images. Then for each $\nu\in E(Y)$, the quadruple $(X,Y,\pi, \nu)$ is a factor quadruple with degree $d$. And the set of all degree joinings over $\nu$ w.r.t. $\pi$ is
\[\{\lambda\in E(\Lambda): \pi_\Lambda \lambda = \nu\}.\] 
\end{theorem}
\begin{proof}
  The same argument as in the proof of the previous theorem shows that the $d$ pre-images of an arbitrary $y\in Y$ is mutually separated. The rest of the proof is similar.
\end{proof}

\begin{remark}
  We remark that a factor code between two irreducible SFTs is constant-to-one if and only if it is bi-closing. Within the class of surjective cellular automata as a special case of factor codes, constant-to-one cellular automata are precisely what is called open cellular automata \cite{kurka2003topological}.
\end{remark}

Next we show that degree joinings for sofic shifts can be obtained from degree joinings for SFTs which are almost one-to-one covers of the original sofic shifts. Recall that each irreducible sofic shift $X$ has an extension $\pi_R: X_R \to X$ where $X_R$ is an irreducible SFT and $\pi_R$ is a factor code that is almost invertible (in this case, equivalent to having degree one). Minimal right-resolving presentations are a special case (see~\cite{LM}).

\begin{theorem}
Let $\pi: X\to Y$ be a finite-to-one factor code on an irreducible sofic shift $X$ with degree $d$. Let $\nu \in E(Y)$ be fully supported. Fix a $\pi_R: X_R\to X$ such that $X_R$ is an irreducible SFT and $\pi_R$ is an almost invertible factor code. Then
\begin{enumerate}
\item $(X, Y, \pi, \nu)$ and $(X_R, Y, \pi\circ\pi_R, \nu)$ are factor quadruples with same degree $d$.
\item For each degree joining $\lambda_R$ for $(X_R, Y, \pi\circ\pi_R, \nu)$, its projection to $X^d$ is a degree joining for $(X, Y, \pi, \nu)$. In other words, if we set $\lambda := (\pi_R)^{\otimes d}(\lambda_R)$, then $\lambda$ is a degree joining for $(X, Y, \pi, \nu)$.
\end{enumerate}
\end{theorem}
\begin{proof}
For each doubly transitive $y\in Y$, by Theorem~\ref{thm:fto-has-degree}, the point $y$ has exactly $d$ pre-images in $X$.
But since $\nu$ is fully supported, by Lemma~\ref{lem:dt-full-in-full}, such $y$ form a full measure set in $Y$ w.r.t. $\nu$. So it follows that $(X, Y, \pi, \nu)$ is a factor quadruple with degree $d$.

The degree of the composition $\pi\circ\pi_R$ is the product of the degree of $\pi$ and that of $\pi_R$ but the degree of $\pi_R$ is one. Therefore $\pi\circ\pi_R$ has degree $d$.
It follows that $(X_R, Y, \pi\circ\pi_R, \nu)$ has the same degree $d$ by a similar argument using doubly transitive points in $Y$.

Next we show that $\lambda$ is a degree joining for the quadruple $(X, Y, \pi, \nu)$.
The measure $\lambda$ is an ergodic measure on $X^d$ because it is an image of an ergodic measure on $(X_R)^d$ under the shift-commuting map $(\pi_R)^{\otimes d}: (X_R)^d \to X^d$.

The map $(\pi_R)^{\otimes d}$ maps the $d$-fold fiber product for $\pi\circ\pi_R$ into a subset of the $d$-fold fiber product for $\pi$. Therefore $\lambda(X_\pi^d)=1$ follows.

Next, we show that $\pi p_1 \lambda = \nu$. With abuse of notation, we write $p_1$ for both the projection $X^d \to X$ to the first component and the projection $(X_R)^d\to X_R$. First, we have $\pi_R \circ p_1 = p_1 \circ (\pi_R)^{\otimes d}$ and this implies that $\pi_R p_1 \lambda_R = p_1\lambda$. Therefore, $\pi (p_1 \lambda) = \pi (\pi_R p_1 \lambda_R) = (\pi \circ \pi_R)(p_1 \lambda_R)$, but this is just $\nu$ because $\lambda_R$ is a degree joining w.r.t. $\pi\circ\pi_R$ over $\nu$. We just showed $\pi p_1 \lambda = \nu$.

To summarize, we showed that $\lambda$ is an ergodic $d$-fold relative joining for the quadruple $(X, Y, \pi, \nu)$ and now we only need to show that it is a separating joining.

To see that $\lambda$ is separating, first notice that $\nu$-a.e. $y$ is doubly transitive, because $\nu$ has full support. Hence, for $\lambda_R$-a.e. $(x^{(1)}, x^{(2)}, \dots, x^{(d)}) \in (X_R)^d$, each $x^{(i)}$ is doubly transitive by Theorem~\ref{thm:fto-and-transitive}, but $\pi_R$ must be injective on doubly transitive points because $\pi_R$ has degree 1. Therefore images in $X$ of $x^{(1)}, x^{(2)}, \dots, x^{(d)}$ under $\pi_R$ are $d$ distinct points, for $\lambda_R$-a.e. $(x^{(1)}, x^{(2)}, \dots, x^{(d)}) \in (X_R)^d$. Since $\lambda$ is defined to be the image of $\lambda_R$ under $(\pi_R)^{\otimes d}$, this shows that $\lambda$ is separating.
\end{proof}

The following two examples show some pathologies when $\nu$ is not fully supported and $\pi$ is not constant-to-one.

\begin{example}\label{ex:diff}
  Let $Y$ be a mixing SFT with some fixed point $y \in Y$ so that $\sigma(y)=y$. By using the blowing-up lemma (Lemma 10.3.2 in \cite{LM}), there exist a mixing SFT $X$ and a finite-to-one factor code $\pi:X \to Y$ such that $\pi^{-1}(y)$ consists of one periodic orbit of least period 2, and every periodic point that is not $y$ has exactly one pre-image under $\pi$.
  Since every periodic point of sufficiently large least period has a unique pre-image, the factor code $\pi$ has degree one. On the other hand, $\nu := \delta_y$ is an ergodic measure that is not fully supported and $d_\nu = 2$. In particular $d_\nu$ exceeds the degree of the factor code.
  The number of ergodic lifts of $\nu$ in this case does not exceed the degree of $\pi$ because the unique invariant measure supported on the periodic orbit of period 2 that maps to $y$ is the unique lift of $\nu$.
\end{example}

\begin{example}
  \label{ex:exceed}
  Let $X$ be a mixing SFT with at least two distinct fixed points $x, x' \in X$. Let $Y=X$ and $y=x$.
  By using Ashley's extension theorem (Theorem 3.15 in \cite{Ashley-Resolving}), there exists a degree one factor code $\pi: X \to Y$ such that $\pi(x) = \pi(x') = y$. The measure $\nu := \delta_y \in E(Y)$ is not fully supported and has at least two different ergodic lifts, namely $\delta_x$ and $\delta_{x'}$. In particular, the number of ergodic lifts of $\nu$ exceeds the degree of $\pi$.
\end{example}

The following example generalizes the above example in order to make the measure $\nu$ less trivial.
\begin{example}
  Let $M, M'$ be irreducible SFTs conjugate to each other and $M\cap M' = \emptyset$. Let $X$ be another irreducible SFT such that $M \cup M' \subset X$. On the image side, let $N=M$ and $Y=X$ so that $N$ is a proper subsystem of $Y$.
  By using Ashley's stronger extension theorem in \cite{Ashley-Extension}, we can extend the obvious two-to-one map $M \cup M' \to N$ to a degree one factor code $\pi: X \to Y$.
  Let $\nu$ be any ergodic measure on $N$. Then $\nu$, seen as an element in $E(Y)$, is not fully supported because $N$ is a proper closed subset of $Y$. The measure $\nu$
  has at least two distinct ergodic lifts $\mu, \mu'$ which are copies of $\nu$ on $M, M'$ respectively.
  In particular, the number of ergodic lifts of $\nu$ exceeds the degree of $\pi$.
\end{example}

We remark that an interesting direction for further research may be to specialize to the problem of lifting finitely described ergodic measures. As a first step in this direction, we raise the following question.
\begin{question}
  Let $\pi: X \to Y$ be a finite-to-one factor code on a mixing SFT. Let $\nu \in E(Y)$ be a hidden Markov measure. Is there an algorithm to decide the number of ergodic lifts of $\nu$? A closely related question is the following. Is there an algorithm to decide which of the margins of a degree joining are equal to which other margins? 
\end{question}

\bibliographystyle{amsplain}
\providecommand{\bysame}{\leavevmode\hbox to3em{\hrulefill}\thinspace}
\providecommand{\MR}{\relax\ifhmode\unskip\space\fi MR }
\providecommand{\MRhref}[2]{%
  \href{http://www.ams.org/mathscinet-getitem?mr=#1}{#2}
}
\providecommand{\href}[2]{#2}

\end{document}